\documentclass [11pt] {article}
\usepackage{amsfonts}
\usepackage{color}
\usepackage{amsthm}
\usepackage{graphicx}
\usepackage{url}
\usepackage{txfonts}
\usepackage{dsfont}
\usepackage{multirow}
\usepackage{booktabs}
\usepackage{setspace}
\renewcommand{\baselinestretch}{1.2}
\newtheorem{theorem}{Theorem}[section]
\newtheorem{proposition}[theorem]{Proposition}
\newtheorem{lemma}[theorem]{Lemma}

\newtheorem{definition}[theorem]{Definition}
\newtheorem{assumption}[theorem]{Assumption}
\newtheorem{algorithm}[theorem]{Algorithm}
\newtheorem{remark}[theorem]{Remark}
\renewcommand{\thefootnote}

\newcommand{\be}{\begin{equation}}
\newcommand{\ee}{\end{equation}}

\begin{document}

\renewcommand{\baselinestretch}{1.2}

\title {Newton's Method for M-Tensor Equations\thanks{Supported by the NSF of China grant No.11771157, 11801184, the NSF of Guangdong Province grant No.2020B1515310013 and the Education Department of Hunan Province grant No.20C0559.}}
\author{Dong-Hui Li \ \ Jie-Feng Xu \ \ and \ \ Hong-Bo Guan\\
School of Mathematical Sciences, \\
South China Normal University, Guangzhou, 510631, China. \\
Email: lidonghui@m.scnu.edu.cn}
\maketitle

\begin{abstract}
We are concerned with the tensor equations whose coefficient tensor is an M-tensor.
We first propose a Newton method for solving the  equation with a positive constant term
and  establish its global and  quadratic convergence.
Then we extend the  method to solve the equation with a nonnegative constant term
and establish its convergence.
At last, we do numerical experiments to test the proposed methods. The results show that the proposed method is
quite efficient.
\end{abstract}


\section{Introduction}

Newton's method is a famous iterative method for solving nonlinear equations.
An attractive property of the method is its superlinear/quadratic convergence
if the Jacobian of the residual function is nonsingular at the solution.
However, when the Jacobian of the function is singular, the
method may loss its fast convergence property or even fail to work.
As an example, we consider the following system of nonlinear equations
\begin{equation}\label{eqn:special}
F_i(x)=\sum_{j=1}^n a_{ij}x_j^{m-1}-b_i=0,\quad i=1,2,\ldots,n.
\end{equation}
When matrix $A=(a_{ij})$ is  nonsingular and $b=(b_1,\ldots,b_n)^T\ge 0$,
the equation has solutions satisfying $\bar x_i=(A^{-1}b)_i^{1/(m-1)}$,
$i=1,2,\ldots,n$. If there is some $\bar x_i=0$, then the Jacobian $F'(\bar x)$ is singular.
As a result, the Newton method may loss its superlinear/quadratic convergence or even be failure.

The last equation (\ref{eqn:special}) is a special tensor equation. In this paper, we
will consider the following general tensor equation
\begin{equation}\label{eqn:M-teq}
F(x)={\cal A}x^{m-1} - b=0,
\end{equation}
where $x,b\in \mathds{R}^n$ and ${\cal A}$ is an ${m}$th-order ${n}$-dimensional tensor
consists of $n^m$ elements:
\[
{\cal A} = (a_{i_1i_2\ldots i_m}), \quad  a_{i_1i_2\ldots i_m}\in \mathds{R}, \quad 1\le i_1, i_2, \cdots, i_m \le n,
\]
and ${{\cal A}x^{m-1}}\in \mathds{R}^n$  with elements
\[
({\cal A}x^{m-1})_i = \sum_{i_2, \ldots, i_m=1}^na_{ii_2\ldots i_m}x_{i_2}\cdots x_{i_m},\quad i=1,2,\ldots,n.
\]

We will pay particular attention to the M-tensor equation (\ref{eqn:M-teq}) in which $\cal A$ is an M-tensor.
To give the definition of M-tensor, we first introduce some concepts.
We refer two recent books \cite{Qi-Chen-Chen-18, Qi-Luo-17} for details.

We denote the set of all ${m}$th-order ${n}$-dimensional tensors by ${\cal T}(m,n)$ and $[n]=\{1,2,\ldots,n\}$.

A tensor ${\cal A}=(a_{i_1i_2\ldots i_m})\in {\cal T}(m,n)$ is called non-negative tensor,
denoted by ${\cal A}\ge 0$, if all its elements are non-negative, i.e., $a_{i_1i_2\ldots i_m}\ge 0$,
$\forall i_1,\ldots, i_m\in [n]$.
${\cal A}$ is called the identity tensor, denoted by ${\cal I}$,
if its diagonal elements are all ones and other elements are zeros,
i.e., all $a_{i_1i_2\ldots i_m}=0$ except $a_{ii\ldots i}=1$, $\forall i, i_1,\ldots, i_m\in [n]$.

If a real number $\lambda$ and a nonzero real vector $x\in \mathds{R}^n$ satisfy
$$
{\cal A}x^{m-1}=\lambda x^{[m-1]},
$$
then $\lambda$ is called an H-eigenvalue of ${\cal A}$ and $x$ is an
H-eigenvector of ${\cal A}$ associated with $\lambda$. Here, for a real scalar $\alpha$,
$x^{[\alpha]}=(x^{\alpha}_1,x^{\alpha}_2,\ldots,x^{\alpha}_n)$
whenever it is meaningful.

A tensor ${\cal A}=(a_{i_1i_2\ldots i_m})\in {\cal T}(m,n)$ is symmetric
if its elements $a_{i_1i_2\ldots i_m}$ are invariant under any permutation of their indices.
The set of all ${m}$th-order ${n}$-dimensional symmetric tensors is denoted by ${\cal ST}(m,n)$.
${\cal A}$ is called semi-symmetric if for any $i\in [n]$, the  sub-tensor
${\cal A}_i:=(a_{ii_2\ldots i_m})_{1\leq i_2,\ldots,i_m\leq n}$ is symmetric.
In the case ${\cal A}\in {\cal ST}(m,n)$, we have
$$
\nabla({\cal A}x^m)=m{\cal A}x^{m-1}.
$$
In the case ${\cal A}\in {\cal T}(m,n)$ is semi-symmetric, we have
$$
\nabla({\cal A}x^{m-1})=(m-1){\cal A}x^{m-2}.
$$

 The definition of M-tensor is introduced in \cite{Ding-Qi-Wei-13,Qi-Luo-17,Zhang-Qi-Zhou-2014}.
\begin{definition}\label{def:Mtensor}
A tensor ${\cal A}\in {\cal T}(m,n)$ is called an M-tensor, if it can be written as
\begin{equation}\label{M-tensor}
{\cal A}=s{\cal I}-{\cal B},\quad {\cal B}\ge 0,\; s\ge \rho ({\cal B}),
\end{equation}
where $\rho ({\cal B})$ is the spectral radius of tensor ${\cal B}$, that is
\[
\rho(\cal B)= \max \left\{\left| \lambda \right|: \lambda \mbox{ is an eigenvalue of } \cal{B}\right \}.
\]
 If $s> \rho ({\cal B})$,
then ${\cal A}$ is called a strong or nonsingular M-tensor.
\end{definition}

For $x,y\in \mathds{R}^n$, we use $x\circ y$ to denote their Hadamard product defined by
\[
x\circ y=(x_1y_1,\cdots,x_ny_n)^T.
\]
We use $\mathds{R}^n_+$ and $\mathds{R}^n_{++}$ to denote the sets of all nonnegative vectors and positive vectors in $\mathds{R}^n$. That is,
\[
\mathds{R}^n_+=\{x\in\mathds{R}^n\,|\,x\geq 0\} \quad {\mbox{and}} \quad \mathds{R}^n_{++}=\{x\in\mathds{R}^n\,|\,x> 0\}.
\]
Tensor equation is also called multi-linear equation. It appears in many practical fields
including data mining and numerical partial
equations \cite{4,Ding-Wei-16,8,Fan-Zhang-Chu-Wei-17,Kressner-Tobler-10,Li-Xie-Xu-17,Li-Ng-15,Xie-Jin-Wei-2017}.

The study in the numerical methods for solving tensor equation has begun only a few years ago.
Most of them focus on solving the M-tensor equation (abbreviated as M-Teq). The existing methods for solving M-Teq focus on
finding a positive solution under the restriction $b>0$. Such as Jacobian and Gauss-Seidal methods \cite{Ding-Wei-16},
the homotopy method \cite{Han-17},
tensor splitting method \cite{Liu-Li-Vong-18},
Newton-type method \cite{He-Ling-Qi-Zhou-18},
continuous time neural network method \cite{Wang-Che-Wei-19}.

Recently, Bai, He, Ling and Zhou \cite{Bai-He-Ling-Zhou-18} proposed a nonnegativity preserving
algorithm to solve M-Teq with $b\geq 0$.
Li, Guan and Wang \cite{Li-Guan-Wang-18} proposed a monotone iterative method to solve the
M-Teq with arbitrary $b$.
Li, Guan and Xu \cite{Li-Guan-Xu-20} proposed an inexact Newton method with $b>0$ and extended
the method to solving the M-Teq with $b\geq 0$.

There are few methods to solve the tensor equation with other structure tensors or more general tensors.
Li, Xie and Xu \cite{Li-Xie-Xu-17} extended the classic splitting methods for solving system of linear equations
to solving tensor equations with symmetric tensor.
Li, Dai and Gao \cite{Dai-2019} proposed a alternating projection method for solving tensor
equations with a special 3-order tensor.
Other related works can also be found in \cite{Berman-Plemmons-94,4,Chang-Pearson-Zhang-11,Gowda-Luo-Qi-Xiu-15,Li-Ng-15,Lim-05,Lv-Ma-18,Qi05, Xie-Jin-Wei-2017a,Xie-Jin-Wei-2017,Huang-18,Yan-Ling-18}.

In this paper, we further study numerical methods for solving M-Teq (\ref{eqn:M-teq}).
Our purpose is to find a nonnegative solution of the equation with $b\geq0$. As we know in \cite{Li-Guan-Xu-20},
finding a nonnegative solution of the M-tensor equation can be done by finding a positive
solution of a lower dimensional M-tensor equation with nonnegative constant term. It is noting that the constant term of
that lower dimensional equation is still not guaranteed to be positive. So most of the
existing methods are not able to be applied. We will propose a Newton
method to get a positive solution of the equation and prove its global convergence and quadratic convergence. Our numerical results
show that the proposed  Newton  method is very efficient.

In the next section, we propose a Newton method to find the unique positive solution to a M-tensor equation with
positive constant term. We will also establish its global and quadratic convergence in Section 2.
In section 3, we extend the idea of the method proposed in Section 2 to get a nonnegative solution of an M-tensor
equation with nonnegative constant term and establish its convergence. It should be pointed out that
such an extension is not trivial because the M-tensor equation with positive and nonnegative constant terms are quite
different. At last, in Section 4, we do numerical experiments to test the proposed methods.

\section{A Newton Method for M-Tensor Equation (\ref{eqn:M-teq}) with $b> 0$}
\setcounter{equation}{0}

In this section, we propose a Newton method to find the unique positive  solution to
(\ref{eqn:M-teq}) with $b>0$.
Throughout this section, without specification, we always suppose that the following
assumption holds.

\begin{assumption}
Tensor $\cal A$
in (\ref{eqn:M-teq}) is a semi-symmetric and strong M-tensor, and $b>0$.
\end{assumption}

Since our purpose is to get a positive solution of the M-Teq (\ref{eqn:M-teq}), we restrict
$x \in \mathds{R}_{++}^{n} .$ Making a variable transformation $y=x^{[m-1]},$ we formulate the
M-Teq (\ref{eqn:M-teq}) as
\begin{equation}\label{equiv-1}
f(y)=F\left(y^{\left[\frac{1}{m-1}\right]}\right)={\cal{A}}\left(y^{\left[\frac{1}{m-1}\right]}\right)^{m-1}-b=0.
\end{equation}
A direct computation gives
\[
f^{\prime}(y)={\cal A}\left(y^{\left[\frac{1}{m-1}\right]}\right)^{m-2}
\mbox{diag}\left(y^{\left[\frac{1}{m-1}-1\right]}\right).
\]
It follows that
\[
f^{\prime}(y) y={\cal{A}}\left(y^{\left[\frac{1}{m-1}\right]}\right)^{m-2}
\mbox{diag}\left(y^{\left[\frac{1}{m-1}-1\right]}\right) y={\cal{A}}\left(y^{\left[\frac{1}{m-1}\right]}\right)^{m-1}=f(y)+b.
\]
For $\epsilon \in(0,1)$ and $\beta>0$, define
$$
{\cal{F}}_{\epsilon}
=\left\{x \in \mathds{R}_{+}^{n}: {\cal{A}} x^{m-1}
\geq \epsilon b\right\}
=\left\{y \in \mathds{R}_{+}^{n}:
{\cal{A}}\left(y^{\left[\frac{1}{m-1}\right]}\right)^{m-1} \geq \epsilon b\right\}
$$
and
$$
\Omega_{\beta}
=\left\{x \in \mathds{R}^{n}:\left\|{\cal A} x^{m-1}-b\right\|
\leq \beta\right\}
=\left\{y \in \mathds{R}^{n}:
\left\|{\cal A}\left(y^{\left[\frac{1}{m-1}\right]}\right)^{m-1}-b \right\| \leq \beta\right\}.
$$
It is easy to see that for any $\epsilon\in (0,1]$, the positive solution of the equation (\ref{eqn:M-teq}) is contained in $\in {\cal{F}}_{\epsilon}$.
In addition, the Jacobian matrices $F^{\prime}(x)$ and $f^{\prime}(y)$
are nonsingular M-matrices for any $x,y \in {\cal{F}}_{\epsilon}$.

\begin{lemma}\label{lem1}
The following statements are true.
\begin{itemize}
\item [](i) If ${\cal{A}}=(a_{i_{1} \ldots i_{m}})$ is a Z-tensor and $b>0$, then for
any $\epsilon>0$, the set ${\cal{F}}_{\epsilon}$ is bounded away from zero. That is,
there is a constant $\delta>0$ such that
$$x \geq \delta \mathbf{e}, \quad \forall x \in {\cal{F}}_\epsilon,$$
where $\mathbf{e}=(1,1, \ldots, 1)^{T}$.
\item [](ii) If ${\cal{A}}$ is a strong M-tensor, then for any $\beta \in \mathds{R}$,
the level set $\Omega_{\beta}$ is bounded.
\end{itemize}
\end{lemma}
\begin{proof}
We prove the lemma by contradiction.

(i) Suppose conclusion (i) is not true. Then, there is a sequence $\{x_k\}\subset {\cal F}_\epsilon$
and an index $i\in [n]$
such that $\{(x_k)_i\}\rightarrow 0.$ Since $x_k\in {\cal F}_\epsilon$, it holds  that
$$
(\epsilon-1)b_i \leq a_{i \ldots i}(x_{k})_{i}^{m-1}+\sum_{(i_{2}, \ldots, i_{m})
\neq(i, \ldots, i)} a_{i i_{2} \ldots i_{m}}(x_{k})_{i_{2}} \ldots(x_{k})_{i_{m}}-b_{i}
\leq a_{i \ldots i}({x_{k}})_{i}^{m-1}-b_{i}.
$$
Taking limits in both sides of the last inequality, we get $\epsilon b_{i} \leq 0.$ It is a contradiction. Consequently, the set $\mathcal{F}_\epsilon$ is bounded away from zero.

(ii) Suppose that for some $\beta \in \mathds{R},$ the level set $\Omega_{\beta}$ is not bounded.
 Then there is a sequence $\left\{x_{k}\right\} \subset \Omega_{\beta}$ satisfying $\|x_k\|\to \infty$,
 as $k\to\infty$. However, we obviously have
$$
\frac{\beta}{\|x_{k}^{[m-1]}\|} \geq \frac{\|{\cal{A}} x_k^{m-1}-b\|}{\|x_{k}^{[m-1]}\|}
\geq \Big \|{\cal A}\Big (\frac{x_{k}}{\|x_{k}\|}\Big )^{m-1}\Big \|-\frac{\|b\|}{\|x_{k}^{[m-1]}\|}.
$$
Suppose that the subsequence $\left\{x_{k} /\left\|x_{k}\right\|\right\}_{K}$ converges to some $\bar{u} \neq 0$.
Taking limits as $k\to\infty$ with $k\in K$ in both sides of the last inequality,
we get ${\cal A} \bar{u}^{m-1}=0 .$ Since $\cal A$ is a strong M-tensor, from Theorem 2.3 in \cite{Li-Guan-Xu-20}, we get a contradiction.
\end{proof}

The idea to develop Newton's method is described as follows. Starting from some $y_0=x_0^{[m-1]}$
satisfying $x_0\in {\cal F}_\epsilon$ with some given small $\epsilon \in (0,1)$, the method generates a sequence
of iterates $\{x_k\}\subset {\cal F}_\epsilon$ by a damped Newton iteration such that the residual
sequence $\{\|f(y_k)\|\}$ is decreasing.

We first show the following lemma.

\begin{lemma}\label{lem2}
Suppose that ${\cal A}$ is a strong M-tensor and $b>0$. Let $d$ be the Newton direction
that is the unique solution of the system of linear equations
\[
f'(y)d+f(y)=0.
\]
Then there is a constant $L>0$ such that the inequality
\[
{\cal A}\left[(y+\alpha d)^{[\frac{1}{m-1}]}\right]^{m-1}
 \geq \epsilon b+\alpha\left((1-\epsilon) b-\frac{1}{2} L \alpha\|d\|^{2} \mathbf{e}\right),
 \quad \forall \alpha>0, \forall y \in \mathcal{F}_{\epsilon} \cap \Omega_{\beta},
\]
where $\textbf{e}=(1,1,\ldots,1)^T.$
\end{lemma}
\begin{proof}
It follows from Lemma \ref{lem1} that the set ${\cal D}={\cal F}_\epsilon \cap \Omega_\beta$ has
positive lower and upper bounds. It is also clear that function $f(y)$ is twice continuously
 differentiable on ${\cal D}$. Denote by $L$ the bound of $\|f''(y)\|$ on ${\cal D}$.
 By the use of the mean-value theorem, we obtain for any $y\in {\cal D}$ and $\alpha>0$,
\begin{eqnarray}\label{feasible}
{\cal A}[(y+\alpha d)^{[\frac{1}{m-1}]}]^{m-1}&=& f(y+\alpha d)+b  \nonumber \\
&=& f(y)+\alpha f^{\prime}(y) d+\alpha \int_{0}^{1}\left[f^{\prime}(y+\alpha \tau d)
-f^{\prime}(y)\right] d \cdot d \tau+b \nonumber  \\
&\geq & \quad(1-\alpha) f(y)+b-\frac{1}{2} L \alpha^{2}\|d\|^{2} \mathbf{e}  \nonumber \\
&=& (1-\alpha) {\cal A}\left[(y)^{\left[\frac{1}{m-1}\right]}\right]^{m-1}+\alpha b-\frac{1}{2}
 L \alpha^{2}\|d\|^{2} \mathbf{e} \nonumber  \\
& \geq &\quad(1-\alpha) \epsilon b+\alpha b-\frac{1}{2} L \alpha^{2}\|d\|^{2} \mathbf{e}  \nonumber \\
&=& \epsilon b+\alpha\left((1-\epsilon) b-\frac{1}{2} L \alpha\|d\|^{2} \mathbf{e}\right).
\end{eqnarray}
The proof is complete.
\end{proof}

Denote
\[
\bar{\alpha}=\min \left\{\frac{2(1-\epsilon) b_{i}}{L\|d\|^{2}}: i \in[n]\right\}.
\]
It follows from the last lemma that if $y \in {\cal D},$ then it holds that
$$
{\cal A}\left[(y+\alpha d)^{\left[\frac{1}{m-1}\right]}\right]^{m-1}
\geq \epsilon b, \quad \forall y \in(0, \bar{\alpha})
$$
The steps of the Newton method are stated as follows.

\begin{algorithm} \label{algo:Newton} ({\bf Newton's Method})

\begin{itemize}
\item [] {\bf Initial.} Given a small constant $\epsilon \in (0,1)$ and constants
$\sigma\in (0,\frac{1}{2}), \eta, \rho\in (0,1)$. Select an initial point $x_0\in {\cal F}_\epsilon$.
Let $y_0=x_0^{[m-1]}$ and $k=0$.
\item [] {\bf Step 1.} Stop if $\|f(y_k)\|\leq \eta$.
\item [] {\bf Step 2.} Solve the system of linear equations
\begin{equation}\label{sub:Newton}
f'(y_k)d_k+f(y_k)=0
\end{equation}
to get $d_k$.
\item [] {\bf Step 3.} Determine a steplength $\alpha _k=\max \{\rho^i:\; i=0,1,\ldots\}$ such that
$y_k+\alpha_kd_k\in {\cal F}_\epsilon$ and that the inequality
\begin{equation}\label{search:descent}
\|f(y_k+\alpha _kd_k)\|^2 \le (1-2 \sigma \alpha _k)\|f(y_k)\|^2
\end{equation}
is satisfied.
\item [] {\bf Step 4.} Let $y_{k+1}=y_k+\alpha _kd_k$ and $x_{k+1}=y_{k+1}^{[\frac{1}{m-1}]}$. Go to Step 1.
\end{itemize}
\end{algorithm}
\noindent
{\bf{Remark}}
\begin{itemize}
\item It is easy to see that the last method is very similar to the standard damped
Newton method except
the line search step where we need to ensure $x_{k+1}\in{\cal F}_\epsilon.$
If $y_k +d_k\in {\cal F}_\epsilon$, then the last method is equivalent to the
standard Newton method for solving nonlinear equation $f(y) = 0.$

\item The steps of the last method ensure that the generated sequence of iterates
$\{x_k\}\in {\cal F}_\epsilon$. As a result,
$f'(y_k)$ is a strong M-matrix and hence the method is well defined. Moreover, the
residual sequence
$\{\|f(y_k)\|\}$ is decreasing. It then follows from Lemma \ref{lem1} that there are
positive constants $c\leq C$ such
that
\begin{equation}
c\textbf{e}\leq y_k \leq C\textbf{e}.
\end{equation}

\item It follows from Lemma \ref{lem2} that if $y_{k} \in {\cal F}_{\epsilon}$, then
$$
y_{k}+\alpha_{k} d_{k} \in {\cal F}_{\epsilon}, \forall \alpha_{k} \in\left(0, \bar{\alpha}_{k}\right)
$$
where
\begin{equation}\label{alpha}
\bar{\alpha}_{k}=\min \left\{\frac{2(1-\epsilon) b_{i}}{L\left\|d_{k}\right\|^{2}}: i \in[n]\right\} \bigcap\left\{-\frac{\left(y_{k}\right)_{i}}{\left(d_{k}\right)_{i}}:\left(d_{k}\right)_{i}<0\right\}
\end{equation}
and $L$ is the bound of $f^{\prime \prime}(y)$ on the set compact set
${\cal{F}}_{\epsilon} \cap \Omega_{\| f\left(y_{0}\right) \|}.$
\end{itemize}

Let $x^*$ be the unique positive solution to the M-Teq and $y^* = (x^*)^{[m-1]}$. It is easy to see that for any
$\epsilon \in (0,1), x^*\in {\cal F}_\epsilon$. Consequently, the matrix $f'(y^*)$
is a nonsingular M-matrix. As a result, the full step
Newton method is locally quadratically convergent.

In what follows, we are going to show that Algorithm \ref{algo:Newton} is globally
convergent and that after finitely
many iterations, the method reduces to the full step Newton method. Consequently, it is quadratically
convergent.
We first show the following lemma.

\begin{lemma}\label{lm:positive}
Suppose that ${\cal A}$ is a strong M-tensor and $b>0$. Then the sequence $\{y_k\}$
and $\{d_k\}$ generated by Algorithm \ref{algo:Newton} are bounded. In addition,
 there is a positive constant $\bar{\alpha}$ such that
 \begin{equation}\label{bound-y}
 y_k+\alpha_kd_k \in {\cal F}_\epsilon,\quad\forall \alpha_k\in (0,\bar{\alpha}).
 \end{equation}
\end{lemma}
\begin{proof}
By the steps of the algorithm, it is easy to see that the sequence $\{y_k\}$ is contained in the compact
set ${\cal D}={\cal F}_\epsilon\cap \Omega_{\|f(y_0)\|}$ and hence bounded. Since $f'(y_k)$
is a nonsingular M-matrix and ${\cal D}$ is compact, the sequence $\{d_k\}_K$ is bounded too.
Notice that $b > 0$ and $\{y_k\}$ has a positive lower bound, the scalar ¡¥$\bar {\alpha}_k$
defined by (\ref{alpha}) has a positive lower bound. This together with the boundedness of
 $\{\|d_k\|\}$ implies that $\bar{\alpha}_k$
has a positive lower bound. Consequently, (\ref{bound-y}) is satisfied with some positive $\bar{\alpha}$.
\end{proof}

The following theorem establishes the global convergence of the proposed method.
\begin{theorem}\label{th:conv-Newton}
Suppose that ${\cal A}$ is a strong M-tensor and $b>0$. Then the sequence of iterates
$\{x_k\}$ generated by Algorithm \ref{algo:Newton} converges to the unique
positive solution to the M-Teq (\ref{eqn:M-teq}).
\end{theorem}
\begin{proof}
It suffices to show that there is an accumulation point $\bar{y}$ of $\{y_k\}$ satisfying $f(\bar{y}) = 0$. Let the
subsequence $\{y_k\}_K$ converge to $\bar{y}$. Without loss of generality, we suppose that the subsequence $\{d_k\}_K$
converges to some $\bar{d}$.

Denote $\tilde{\alpha}=\liminf _{k \rightarrow \infty, k \in K} \alpha_{k} .$ If $\tilde{\alpha}>0,$
the inequality (\ref{search:descent}) implies $f(\bar{y})=0 .$ Consider the case $\tilde{\alpha}=0 .$
By the line search rule, when $k \in K$ is sufficiently large, $\alpha_{k}^{\prime}=\rho^{-1} \alpha_{k}$
will not satisfy $(\ref{search:descent}),$ i.e.,
$$
\left\|f\left(y_{k}+\alpha_{k}^{\prime} d_{k}\right)\right\|^{2}-\left\|f\left(y_{k}\right)\right\|^{2}>-2 \sigma \alpha_{k}^{\prime}\left\|f\left(y_{k}\right)\right\|^{2}
$$
Dividing both sizes of the last inequality by $\alpha_{k}^{\prime}$ and then taking limits as
$k \rightarrow \infty$ with $k \in K,$ we get
\begin{equation}\label{fy}
2 f(\bar{y})^{T} f^{\prime}(\bar{y}) \bar{d} \geq-2 \sigma\|f(\bar{y})\|^{2}
\end{equation}
On the other hand, by taking limits in (\ref{sub:Newton}), we can obtain
$f^{\prime}(\bar{y}) \bar{d}+f(\bar{y})=0 .$ It together with (\ref{fy}) and
the fact $\sigma \in(0,1)$ yields $f(\bar{y})=0$.
The proof is complete.
\end{proof}

The last theorem established the global convergence of the proposed Newton method.
Moreover, we see from (\ref{feasible}) that $x_k+d_k\in {\cal F}_{\epsilon}$ for all $k$ sufficiently large
because $\{d_k\}\to 0$. Consequently, the method locally reduces to a standard damped Newton method.
Following a
standard discussion as the proof of the quadratic convergence of a damped Newton method, it is not
difficult to prove the quadratic convergence of the method. We give the result but omit the proof.

\begin{theorem}\label{th:conv-Newton2}
Let the conditions in Theorem \ref{th:conv-Newton} hold.
Then the convergence rate of the sequence $\{y_k\}$ generate by Algorithm \ref{algo:Newton} is quadratic.
\end{theorem}

\section{An Extension}
\setcounter{equation}{0}
In this section, we extend the Newton method proposed in the last section to the M-Teq (\ref{eqn:M-teq}) with $b\geq 0$. In
the case $b$ has zero elements, the M-Teq may have multiple nonnegative or positive solutions. Our purpose
is to find one nonnegative or positive solution of the equation.
By Theorem 2.6 in \cite{Li-Guan-Xu-20}, a nonnegative solution
of (\ref{eqn:M-teq}) has zero elements if and only if ${\cal A}$ is reducible with respect to some $I\subseteq I_0$, where
\[
I_0=\{i\in [n]\;|\; b_i=0\}.
\]
Since justifying the reducibility is an easy task, without loss of generality, we suppose that the nonnegative solutions of the M-Teq are positive.

\begin{assumption}\label{ass:B}
Suppose $b\ge 0$ and that tensor $\cal A$ is a strong M-tensor and irreducible with respect to $I_0$.
Suppose further that for each $i\in I_0$, there is an element
$a_{ii_2\ldots i_m}\neq 0$ with, $i_2,\ldots, i_m\in I_+$.
\end{assumption}

Under the conditions of Assumption \ref{ass:B}, we have
\[
{\cal A}^{c}(y^{[\frac{1}{m-1}.})^{m-1})_{I_{0}}<0,\quad \forall y\in \mathbb R^n_{++}.
\]

For the sake of convenience, we introduce some notations. Denote $I_{+}=\{i: b_{i}>0\}$  and $I_{0}=\{i: b_{i}=0\}$.
For given constants $1>\epsilon>\epsilon'>0$, we define
\[
\overline{{\cal F}}_{\epsilon, \epsilon'}=\overline{{\cal F}}_{\epsilon}^1 \cap \overline{{\cal F}}_{\epsilon '}^2
\]
with
\[
\overline{{\cal F}}_{\epsilon}^1= \Big \{y \in\mathbb R_{++}^{n}:   \Big ({\cal A} \Big (y^{[\frac{1}{m-1}]} \Big )^{(m-1)} \Big )_{I_+}\geq \epsilon b_{I_{+}} \Big  \}
\]
and
\[
\overline{{\cal F}}_{\epsilon '}^2=  \Big \{y\in \mathbb R^n_{++}:\;
 \Big ({ \cal A} \Big (y^{[\frac{1}{m-1}]} \Big )^{m-1} \Big )_{I_{0}} \geq \epsilon' f'(y)_{I_0I_+}f'(y)_{I_+I_+}^{-1}b_{I_+} \Big \}.
\]

It is easy to see that every solution $\bar x\in \overline{{\cal F}}_{\epsilon}$.

For $y\in \mathbb R^n_{++}$, we split $f'(y)$ into
\[
f'(y)=\left (\begin{array}{cc}
f'_{I_+I_+}(y) & f'_{I_+I_0}(y) \\ f'_{I_0I_+}(y) & f'_{I_0I_0}(y)
\end{array} \right ).
\]
It is easy to see that $f'(y)$ is a Z-matrix.

The next theorem shows that for any $y\in \overline{{\cal F}}_{\epsilon}$, $f'(y)$ is a nonsingular M-matrix.
As a result, the set $\overline{{\cal F}}_{\epsilon}$ is well defined.

\begin{theorem}
Let $1>\epsilon>\epsilon'>0$. For any $y\in \overline{{\cal F}}_{\epsilon,\epsilon'}$, $f'(y)$ is a nonsingular M-matrix.
\end{theorem}
\begin{proof}
By direct computation, we get $f'(y)y={\cal A}\Big ( y^{[\frac 1{m-1}]} \Big )^{m-1}$.
The condition $y\in \overline {\cal F}_{\epsilon} ^1$ yields
\[
0<\epsilon b_{I_+} \le \Big ( {\cal A}\Big ( y^{[\frac 1{m-1}]} \Big )^{m-1} \Big )_{I_+} = f'(y)_{I_+I_+} y_{I_+}+f'(y)_{I_+I_0}y_{I_0}
\le f'(y)_{I_+I_+} y_{I_+}.
\]
Consequently, $f'(y)_{I_+I_+}$ is a nonsingular M-matrix.
We are going to show that the Schur complement
\[
 f'(y)_{I_0I_0}- f'(y)_{I_0I_+} f'(y)_{I_+I_+}^{-1} f'(y)_{I_+I_0}
 \]
is also a nonsingular M-matrix.

Observing $f'(y)_{I_0I_+}\le 0$, we get from the condition $y\in \overline {\cal F}^2_{\epsilon'}$
\begin{eqnarray*}
f'(y)_{I_0I_0}y_{I_0} &\ge & \epsilon' f'(y)_{I_0I_+}f'(y)_{I_+I_+}^{-1}b_{I_+} - f'(y)_{I_0I_+}y_{I_+} \\
    & \ge & \epsilon ' f'(y)_{I_0I_+}f'(y)_{I_+I_+}^{-1}b_{I_+} - f'(y)_{I_0I_+} f'(y)_{I_+I_+}^{-1}\Big (\epsilon b_{I_+} -f'(y)_{I_+I_0} y_{I_0} \Big )  \\
    & = & -(\epsilon-\epsilon' ) f'(y)_{I_0I_+} f'(y)_{I_+I_+}^{-1}b_{I_+} + f'(y)_{I_0I_+} f'(y)_{I_+I_+}^{-1} f'(y)_{I_+I_0}y_0,
\end{eqnarray*}
which implies
\[
\Big ( f'(y)_{I_0I_0}- f'(y)_{I_0I_+} f'(y)_{I_+I_+}^{-1} f'(y)_{I_+I_0}\Big  )y_{I_0}
\ge  -(\epsilon -\epsilon') f'(y)_{I_0I_+} f'(y)_{I_+I_+}^{-1}b_{I_+}>0.
\]
The last condition ensures that $f'(y)$ is a nonsingular M-matrix.
\end{proof}

Similar to Lemma \ref{lem1}, we have the following lemma.
\begin{lemma}
If ${\cal A}=\left(a_{i_{1} \ldots i_{m}}\right)$ is a Z-tensor and $b \geq 0,$ then for any
$\epsilon>0,$ the set $\overline{{\cal F}}_{\epsilon}$ is bounded away from zero. That is, there is a constant $\delta>0$ such that
$$
y \geq \delta \mathbf{e}, \quad \forall y \in \overline{{\cal F}}_{\epsilon,\epsilon'}.
$$
\end{lemma}
\begin{proof}
First, following the same arguments as the proof of Lemma \ref{lem1} (i), it is easy to show that the
elements $y_{I_+}$ has a positive lower bound. We only need to prove that $y_{I_0}$ has a positive lower bound too.

It is easy to see by the definition of $\overline {\cal F}_{\epsilon} ^1$ that  each $y\in \overline {\cal F}_{\epsilon} ^1$
satisfies
\[
\epsilon b_{I_+}\le   \Big ({\cal A} \Big (y^{[\frac{1}{m-1}]} \Big )^{(m-1)} \Big )_{I_+} = (f'(y)y)_{I_+} \le f'(y)_{I_+I_+}y_{I_+},
\]
which implies
\[
y_{I_+}\ge f'(y)_{I_+I_+}^{-1}b_{I_+}>\epsilon' f'(y)_{I_+I_+}^{-1} b_{I_+}.
\]

The condition $y\in \overline {\cal F}^2_{\epsilon'}$ implies
\[
0\le f'_{I_0I_+}(y) \Big (y_{I_+} - \epsilon' f'(y)_{I_+I_+}^{-1} b_{I_+}\Big )+ f'_{I_0I_0}(y) y_{I_0}
\]
By the condition of Assumption \ref{ass:B} and the fact that $y_{I_+}$ has positive lower bound, we claim
that  the vector $f'_{I_0I_+}(y) \Big (y_{I_+} - \epsilon' f'(y)_{I_+I_+}^{-1} b_{I_+}\Big )$ is bounded away from zero.
Taking into account that  $f'_{I_0I_0}(y)$ is a Z-matrix and $y_{I_0}>0$, it is easy to see that
$y_{I_0}$ is bounded away from zero too.

The proof is complete.
\end{proof}

In what follows, we propose a Newton method for finding a positive solution to the M-Teq (\ref{eqn:M-teq}) with
$b\geq 0$ as follows

\begin{algorithm}{\bf (Extended Newton Method for (\ref{eqn:M-teq}) with $b\geq 0$) }\label{algo:extension}
\begin{itemize}
\item [] {\bf Initial.} Given  constants $\epsilon, \epsilon',\rho, \eta, \sigma\in (0,1)$ satisfying $\epsilon'<\epsilon$.
Select an initial point $y_0\in\overline{{\cal F}}_{\epsilon,\,\epsilon'}$. Let  $k=0$.
\item [] {\bf Step 1.} Stop if $\|f(y_k)\|\leq \eta$.
\item [] {\bf Step 2.} Solve the system of linear equations
\begin{equation}\label{sub:extension}
f'(y_k)d_k+f(y_k)=0.
\end{equation}
to get $d_k$.
\item [] {\bf Step 3.} Determine a steplength $\alpha _k=\max \{\rho^i:\; i=0,1,\ldots\}$ such that
$y_k+\alpha_kd_k\in \overline{{\cal F}}_{\epsilon,\,\epsilon'}$ and that the inequality
\begin{equation}\label{search:descent1}
\|f(y_k+\alpha _kd_k)\|^2 \le (1-2 \sigma \alpha _k)\|f(y_k)\|^2
\end{equation}
is satisfied.
\item [] {\bf Step 4.} Let $y_{k+1}=y_k+\alpha _kd_k$. Go to Step 1.
\end{itemize}
\end{algorithm}

In what follows, we show that the algorithm above is well-defined.
\begin{proposition}
Let the conditions in Assumption \ref{ass:B} hold. Then Algorithm \ref{algo:extension}
is well defined.
\end{proposition}
\begin{proof} It suffices to
verify that the relation $y_k+\alpha_kd_k\in \overline{{\cal F}}_{\epsilon,\,\epsilon_{k+1}'}$
is satisfied for all $\alpha>0$ sufficiently small. Indeed, we have
\begin{eqnarray*}
\lefteqn {
{\cal A} \Big ( (y_k+\alpha d_k)^{[\frac{1}{m-1}]} \Big )^{m-1}  =  f(y_k+\alpha d_k) +b }\\
    &= & f(y_k)+\alpha f'(y_k)d_k + O( \|\alpha d_k \|^2)+b \\
    &=& (1-\alpha )f(y_k) +b +O( \|\alpha d_k \|^2) \\
    &=& (1-\alpha) {\cal A} \Big ( y_k^{[\frac{1}{m-1}]} \Big )^{m-1} + \alpha b +O( \|\alpha d_k \|^2) .
\end{eqnarray*}
Since $y_k\in \overline {\cal F}^1_\epsilon$ and $b_{I_+}>0$, we get from the last equality
\begin{equation}\label{steplength-1}
\Big ( {\cal A} \Big ( (y_k+\alpha d_k)^{[\frac{1}{m-1}]} \Big )^{m-1} \Big )_{I_+}  
\ge \epsilon b_{I_+}  +\alpha [(1-\epsilon )b_{I_+} +O( \alpha \|d_k \|^2),
\end{equation}
which implies $y_k+\alpha d_k\in \overline {\cal F}^1_{\epsilon}$ for all $\alpha>0$ sufficiently small.

Similarly, we have by the fact $y_k\in \overline {\cal F}^2_{\epsilon'}$ and $b_{I_0}=0$
\begin{equation}\label{steplength-2}
\Big ( {\cal A} \Big ( (y_k+\alpha d_k)^{[\frac{1}{m-1}]} \Big )^{m-1} \Big )_{I_0}
\ge \epsilon' f'(y_k)_{I_0I_+}f'(y_k)_{I_+I_+}^{-1} b_{I_+} + \alpha \Big (- \epsilon' f'(y_k)_{I_0I_+}f'(y)_{I_+I_+}^{-1} b_{I_+} +O( \alpha \|d_k \|^2)\Big ).
\end{equation}
By the condition of Assumption \ref{ass:B}, it is not difficult to see from the last
inequality that we claim that the inequality
\[
\Big ( {\cal A} \Big ( (y_k+\alpha d_k)^{[\frac{1}{m-1}]} \Big )^{m-1} \Big )_{I_0}\ge \epsilon' f'(y)_{I_0I_+}f'(y)_{I_+I_+}^{-1} b_{I_+}
\]
is satisfied for all $\alpha>0$ sufficiently small.
\end{proof}

It is easy to show that Lemma \ref{lem1} holds true for the case $b\ge 0$. As a result, the sequence generated by
Algorithm \ref{algo:extension} is bounded. Consequently, the inequalities (\ref{steplength-1}) and (\ref{steplength-2})
ensure that there is a positive constant $\bar\alpha>0$ such that $x_k+\alpha d_k\in \overline {\cal F}_{\epsilon,\epsilon'}$.
$\forall \alpha \in (0,\bar\alpha]$.

Similar to the proof of Theorem \ref{th:conv-Newton}, we can prove the global convergence of Algorithm  \ref{algo:extension}.

\begin{theorem}\label{th:conv-e}
Let the conditions in Assumption \ref{ass:B} hold. Then the sequence of iterates $\left\{y_{k}\right\}$ generated by
Algorithm \ref{algo:extension} is bounded. Moreover, every accumulation point of the iterates $\{y_k\}$ is a positive solution to the M-tensor equation $f(y)=0 .$
\end{theorem}

The remainder of this section is devoted to the proof of the quadratic convergence of Algorithm  \ref{algo:extension}.
It should be pointed out that the unit steplength may not be acceptable due to the existence of zero elements in $b$.
To ensure the quadratic convergence of the method, we need to make a slight modification to Step 3 of the algorithm. Specifically,
we use the following Step 3${}'$ instead of Step 3 in Algorithm  \ref{algo:extension}.

{\bf Step 3${}'$.} If $\alpha_k=1$ satisfies
$y_k+\alpha_kd_k\in \overline{{\cal F}}_{\epsilon,\,\epsilon'}$ and (\ref{search:descent1}),
then we let $\alpha_k=1$. Otherwise,  for given constant $c>0$, we let $\beta_k=1-c\|f(y_k)\|$. If $\beta_k\le 0$, we let $\beta_k=1$.
Determine a steplength $\alpha _k=\max \{\beta _k\rho^i:\; i=0,1,\ldots\}$ such that
$y_k+\alpha_kd_k\in \overline{{\cal F}}_{\epsilon,\,\epsilon'}$ and that the inequality (\ref{search:descent1})
is satisfied.

It is not difficult to see that the global convergence still remains true if Step 3 is replaced by Step 3${}'$.
Moreover, since $\{d_k\}\to 0$, it is easy to prove from (\ref{steplength-1}), (\ref{steplength-2})
and (\ref{search:descent1})
that for all $k$ sufficiently large, the step $\alpha_k=\beta_k=1-c\|f(y_k)\|$ will be accepted.
In this case, the sequence of iterates $\{y_k\}$ satisfies
$y_{k+1}=y_k+\bar d_k$, with $\bar d_k= \beta _kd_k = (1-c \|f(y_k)\|)d_k$
satisfying
\[
f'(y_k)\bar d_k + f(y_k)= f'(y_k)d_k +f(y_k) -c\|f(y_k)\|d_k= -c\|f(y_k)\|d_k.
\]
If $y_{k+1}=y_k+\bar d_k\in \overline{\cal F}_{\epsilon,\epsilon'}$, then
 when $k$ is sufficiently, $y_k$ can be regarded as the sequence generated by
a full step inexact Newton method. Consequently, the quadratic convergence becomes well-known. 

\begin{theorem}
Let the conditions in Assumption \ref{ass:B} hold. Suppose that the sequence of iterates $\left\{y_{k}\right\}$ generated by
Algorithm \ref{algo:extension} converges to a positive solution $y^*$ to the M-tensor equation $f(y)=0 .$  Then the convergence
rate of $\{y_{k} \}$ is quadratic.
\end{theorem}
\begin{proof}
We only need to verify
\begin{equation}\label{unit-2}
y_{k+1}=y_k+\bar d_k= y_k+ \Big (1-c \|f(y_k)\|\Big )d_k\in \overline{\cal F}_{\epsilon,\epsilon'}.
\end{equation}

It is not difficult to show from (\ref{sub:extension}) that
\[
\|d_k\| = O(\|f(y_k)\|)=O(\|f(y_k)\|)=O(\|x-x^*\|).
\]
Similar to the proof of (\ref{steplength-1}), we can derive
\[
\Big ( {\cal A} \Big ( (y_k+\beta _k d_k)^{[\frac{1}{m-1}]} \Big )^{m-1} \Big )_{I_+}
\ge \epsilon b_{I_+}  +\beta_k [(1-\epsilon )b_{I_+} +O( \beta_k \|d_k \|^2).
\]
Since $\{\beta_k\}\to 1$ and $\{d_k\}\to 0$, the last inequality implies $y_k+\beta_kd_k\in \overline{\cal F}^1_{\epsilon}$.

We also can obtain
\[
\Big ( {\cal A} \Big ( (y_k+\beta_k d_k)^{[\frac{1}{m-1}]} \Big )^{m-1} \Big )_{I_0}
\ge \epsilon' f'(y_k)_{I_0I_+}f'(y_k)_{I_+I_+}^{-1} b_{I_+} + \beta_k  \Big (- \epsilon' f'(y_k)_{I_0I_+}f'(y_k)_{I_+I_+}^{-1} b_{I_+} +O( \beta_k \|d_k \|^2)\Big ).
\]
By the condition of Assumption \ref{ass:B}, it is clear that $f'(y_k)_{I_0I_+}f'(y_k)_{I_+I_+}^{-1} b_{I_+} <0$. Consequently,
the last inequality implies $y_k+\beta_kd_k\in \overline{\cal F}^2_{\epsilon}$.
The proof is complete.
\end{proof}

\section{Numerical Results}
\setcounter{equation}{0}
In this section, we do numerical experiments to test the effectiveness of the proposed methods.
We implemented our methods in Matlab R2019a and ran the codes on a computer with Intedl(R) Core(TM) i7-10510U CPU @ 1.80GHz 2.30 GHz and 16.0 GB RAM.
We used a tensor toolbox \cite{Bader-Kolda-15} to proceed some tensor computation.

The test problems are from \cite{Ding-Wei-16,Li-Guan-Xu-20,Li-Xie-Xu-17,Xie-Jin-Wei-2017}.

{\bf Problem 1.} We solve tensor equation (\ref{eqn:M-teq}) where ${\cal A}$ is a symmetric strong M-tensor of order $m$ $(m=3,4,5)$ in the form ${\cal A}=s{\cal I}-{\cal B}$, where tensor ${\cal B}$ is symmetric tensor
whose entries are uniformly distributed in $(0,1)$, and
\[
s=(1+0.01)\cdot\max_{i=1,2,\ldots,n}({\cal B}{\bf e}^{m-1})_i,
\]
where ${\bf e}=(1,1,\ldots,1)^T$.

{\bf Problem 2.} We solve tensor equation (\ref{eqn:M-teq}) where $\cal A$ is a symmetric strong M-tensor of order $m$ $(m=3,4,5)$ in the form ${\cal A}=s{\cal I}-{\cal B}$, and tensor ${\cal B}$ is a nonnegative tensor with
\[
b_{i_1i_2\ldots i_m}=|\mathrm{sin}(i_1+i_2+\ldots +i_m)|,
\]
and $s=n^{m-1}$.

{\bf Problem 3.} Consider the ordinary differential equation
$$
\frac{d^2x(t)}{dt^2}=-\frac{GM}{x(t)^2},\quad t\in(0, 1),
$$
with Dirichlet's boundary conditions
$$
x(0)=c_0, \quad x(1)=c_1,
$$
where $G\approx 6. 67\times 10^{-11}Nm^2/kg^2$ and $M\approx 5. 98\times 10^{24}$ is the gravitational constant and the mass of the earth.

Discretize the above equation, we have
\[\left\{
\begin{array}{l}
	x^3_1=c^3_0, \\
	2x^3_i-x^2_ix_{i-1}-x^2_ix_{i+1}=\frac{GM}{(n-1)^2}, \quad i=2, 3, \cdots, n-1, \\
	x^3_n=c^3_1.
\end{array}
\right.
\]
It is a tensor equation, i.e.,
$${\cal A}x^3=b, $$
where ${\cal A}$ is a 4-th order M tensor whose entries are
\[
\left\{
\begin{array}{l}
	a_{1111}=a_{nnnn}=1, \\
	a_{iiii}=2, \quad i=2, 3, \cdots,  n-1, \\
	a_{i(i-1)ii}=a_{ii(i-1)i}=a_{iii(i-1)}=-1/3, \quad i=2, 3, \cdots,  n-1, \\
	a_{i(i+1)ii}=a_{ii(i+1)i}=a_{iii(i+1)}=-1/3, \quad i=2, 3, \cdots,  n-1, \\
\end{array}
\right.
\]
and b is a positive vector with
\[
\left\{\begin{array}{l}
	b_1=c^3_0, \\
	b_i=\frac{GM}{(n-1)^2}, \quad i=2, 3, \cdots,  n-1, \\
	b_n=c^3_1.
\end{array}\right.
\]

{\bf Problem 4.} We solve tensor equation (\ref{eqn:M-teq}) where $\cal A$ is a non-symmetric strong M-tensor of order $m$ $(m=3,4,5)$ in the form ${\cal A}=s{\cal I}-{\cal B}$, and tensor ${\cal B}$ is nonnegative
tensor whose entries are uniformly distributed in $(0,1)$.
The parameter $s$ is set to
\[
s=(1+0.01)\cdot\max_{i=1,2,\ldots,n}({\cal B}{\bf e}^{m-1})_i.
\]

{\bf Problem 5.} We solve tensor equation (\ref{eqn:M-teq}) where $\cal A$ is a lower triangle strong M-tensor of order $m$ $(m=3,4,5)$ in the form ${\cal A}=s{\cal I}-{\cal B}$, and tensor ${\cal B}$ is a strictly lower triangular nonnegative
tensor whose entries are uniformly distributed in $(0,1)$.
The parameter $s$ is set to
\[
s=(1-0.5)\cdot\max_{i=1,2,\ldots,n}({\cal B}{\bf e}^{m-1})_i.
\]

For Problem 1, 2 4 and 5, similar to \cite{Han-17,He-Ling-Qi-Zhou-18}, we solved the tensor equation
\[
\hat{F}(x)=\hat{\cal {A}}x^{m-1}-\hat{b}=0
\]
instead of the tensor equation (\ref{eqn:M-teq}),
where $\hat{\cal{A}}:=\cal{A}/\omega$ and $\hat{b}:=b/\omega$ with $\omega$ is the largest value among
the absolute values of components of $\cal{A}$ and $b$. The stopping criterion  is set to
\[
\|\hat{F}(x_k)\|\leq 10^{-10}.
\]
And for Problem 3, the stopping criterion  is set to
$$
\frac{\|{\cal A}x^{m-1}-b\| }{\|b\|}\leq 10^{-10}.
$$
We also stop the tested algorithms if the number of iteration reaches to 300,  which means that the method is failure for the problem.

\begin{remark}\label{rm:initial}
	Since ${\cal A}$ is a strong M-tensor, there exists a positive vector $u$ such that ${\cal A}u^{m-1}>0$.
	This vector $u$ can be obtained in a certain iteration of solving ${\cal A}x^{m-1}= {\bf e}$ by the existing methods proposed in \cite{Ding-Wei-16,Han-17, He-Ling-Qi-Zhou-18}.
	Then we can get an initial point of Algorithm 2.4 or Algorithm 3.4 by letting $x_{0} = t {u}$ and $y_{0} = x_{0}^{[m-1]}$ with sufficient large constant $t$.
	Particularly, if $\cal A$ is a diagonally dominant M-tensor we can simply let $u = {\bf e}$.
\end{remark}

Note that strong M-tensor constructed in Problems 1, 2 and 4 are diagonally dorminant M-tensor.
Whereas Problem 3 and 5 both are non-diagonally dorminant M-tensor.

We first test the performance of Algorithm 2.4 (Newton's Method denoted by 'NM').
In order to test the effectiveness of the proposed method, we compare the Newton method with Inexact Newton Method (denoted by 'INM') proposed in \cite{Li-Guan-Xu-20}.
We take the parameter of NM be $\epsilon=0.1, \sigma=0.1$ and $\rho=0.5$.
And let parameters of INM be $\sigma=0.1, \rho=0.4$.
We set the initial point for INM as in \cite{Li-Guan-Xu-20}, i.e., $y_{0} = t{\bf e}>0$ such that $f(y_{0}) \le b$, where $t$ is a sufficient small positive constant.
We use INM to find an initial point of NM for Problem 3 and 5, i.e., in the iterative of INM, if ${\cal A}x_{k}^{m-1}>0$, let the vector $u=x_{k}$ in Remark \ref{rm:initial}.

For the stability of numerical results, we test the problems of different sizes.
For each pair $(m, n)$, we randomly generate 50 tensors $\cal A$ and $b\in (0, 1)$.
The results are listed in Tables \ref{tb:NM-INM-p1} to \ref{tb:NM-INM-p5}, thereinto 'Pro' represent the test problem;
'Iter' represents the average number of iterations;
'Time-Int' denotes the average time to find an initial point of NM;
'Time' denotes the computing time (in seconds) including initial time to find an approximate solution.
And the ratio signs are denoted bellow.
\begin{table}[h!]
	\centering
	\begin{tabular}{cccc}
		\toprule
		R-Int & RI & RT & RT1\\
		\hline
		\specialrule{0em}{1pt}{1pt}
		$\frac{\mbox{Time-Int}}{\mbox{Time of NM}}$ & $\frac{\mbox{Iter of NM}}{\mbox{Iter of INM}}$
		& $\frac{\mbox{Time of NM}}{\mbox{Time of INM}}$ & $\frac{\mbox{Time of NM minus Time-Int} }{\mbox{Time of INM}}$\\
		\specialrule{0em}{1pt}{1pt}
		\bottomrule
	\end{tabular}
\end{table}

\renewcommand\tabcolsep{5.0pt}
{
	\begin{table}[h!]
		\centering
		\begin{tabular}{c|cc|cc|cc}\toprule
			\multirow{2}{*}{}
			& \multicolumn{2}{c|}{NM} & \multicolumn{2}{c|}{INM} & \multicolumn{2}{c}{Rate} \\
			$(m,n)$ & Iter& Time & Iter & Time & RI & RT\\
			\hline
			(3,200)  &2 & 0.01163   &11 & 0.03634  & $18.1\%$ & $32.0\%$\\
			(3,401)  &2 & 0.08120   &12 & 0.39976  & $16.9\%$ & $20.3\%$\\
			(3,650)  &2 & 0.39719   &13 & 1.79962  & $15.8\%$ & $22.1\%$\\
			\hline
			(4, 40)  &2 & 0.00366   &8.4 & 0.00727  & $23.8\%$ & $50.4\%$\\
			(4, 71)  &2 & 0.02786   &9.2 & 0.10398  & $21.7\%$ & $26.8\%$\\
			(4,100)  &2 & 0.11595   &9.6 & 0.45303  & $20.9\%$ & $25.6\%$\\
			(4,130)  &2 & 0.38220   &10 & 1.34639  & $19.6\%$ & $28.4\%$\\
			\hline
			(5, 30)  &2 & 0.02681   &7.6 & 0.08171  & $26.4\%$ & $32.8\%$\\
			(5, 48)  &2 & 0.37926   &8.3 & 1.27908  & $24.0\%$ & $29.7\%$\\
			\bottomrule
		\end{tabular}
		\caption{\small  Comparison on Problem 1.}\label{tb:NM-INM-p1}
	\end{table}
}
{
	\begin{table}[h!]
		\centering
		\begin{tabular}{c|cc|cc|cc}\toprule
			\multirow{2}{*}{}
			& \multicolumn{2}{c|}{NM} & \multicolumn{2}{c|}{INM} & \multicolumn{2}{c}{Rate} \\
			$(m,n)$ & Iter& Time & Iter & Time & RI & RT\\
			\hline
			(3,200)  &3 & 0.01214   &11 & 0.03416  & $28.2\%$ & $35.5\%$\\
			(3,401)  &3 & 0.10887   &12 & 0.40037  & $25.3\%$ & $27.2\%$\\
			(3,650)  &3 & 0.48992   &12 & 1.94355  & $24.9\%$ & $25.2\%$\\
			\hline
			(4, 40)  &3 & 0.00353   &9.4 & 0.00726  & $31.8\%$ & $48.6\%$\\
			(4, 71)  &3 & 0.03599   &9.3 & 0.11561  & $32.2\%$ & $31.1\%$\\
			(4,100)  &2.7 & 0.14569   &9.1 & 0.52564  & $29.4\%$ & $27.7\%$\\
			(4,130)  &2 & 0.35740   &9.7 & 1.59151  & $20.9\%$ & $22.5\%$\\
			\hline
			(5, 30)  &2.4 & 0.02937   &8.4 & 0.09164  & $29.0\%$ & $32.1\%$\\
			(5, 48)  &2 & 0.38996   &8 & 1.61044  & $24.9\%$ & $24.2\%$\\
			\bottomrule
		\end{tabular}
		\caption{\small  Comparison on Problem 2.}\label{tb:NM-INM-p2}
	\end{table}
}
{
	\begin{table}[h!]
		\centering
		\begin{tabular}{c|cccc|cc|ccc}\toprule
			\multirow{2}{*}{}
			& \multicolumn{4}{c|}{NM} & \multicolumn{2}{c|}{INM} & \multicolumn{3}{c}{Rate} \\
			$(m,n)$ & Iter& Time &Time-Int & R-Int & Iter & Time & RI & RT & RT1\\
			\hline
			(4, 40)  &1 & 0.00263 & 0.00166  & $63.1\%$     &5 & 0.00495      & $20.0\%$ & $53.1\%$ & $19.6\%$ \\
			(4, 71)  &1 & 0.03952 & 0.02627  & $66.5\%$     &5 & 0.07993      & $20.0\%$ & $49.4\%$ & $16.6\%$ \\
			(4,100)  &1 & 0.16575 & 0.11177  & $67.4\%$     &5 & 0.33503      & $20.0\%$ & $49.5\%$ & $16.1\%$ \\
			(4,130)  &1 & 0.45219 & 0.29926  & $66.2\%$     &5 & 0.91583      & $20.0\%$ & $49.4\%$ & $16.7\%$ \\
			\bottomrule
		\end{tabular}
		\caption{\small  Comparison on Problem 3.}\label{tb:NM-INM-p3}
	\end{table}
}
{
	\begin{table}[h!]
		\centering
		\begin{tabular}{c|cc|cc|cc}\toprule
			\multirow{2}{*}{}
			& \multicolumn{2}{c|}{NM} & \multicolumn{2}{c|}{INM} & \multicolumn{2}{c}{Rate} \\
			$(m,n)$ & Iter& Time & Iter & Time & RI & RT\\
			\hline
			(3,200)  &2 & 0.00886   &11 & 0.03518  & $17.9\%$ & $25.2\%$\\
			(3,401)  &2 & 0.08087   &12 & 0.39426  & $17.2\%$ & $20.5\%$\\
			(3,650)  &2 & 0.36240   &12 & 1.97655  & $16.4\%$ & $18.3\%$\\
			\hline
			(4, 40)  &2 & 0.00272   &8.7 & 0.00662  & $22.9\%$ & $41.1\%$\\
			(4, 71)  &2 & 0.02710   &9.4 & 0.11227  & $21.3\%$ & $24.1\%$\\
			(4,100)  &2 & 0.11558   &9.4 & 0.54760  & $21.4\%$ & $21.1\%$\\
			(4,130)  &2 & 0.35507   &9.9 & 1.57429  & $20.3\%$ & $22.6\%$\\
			\hline
			(5, 30)  &2 & 0.02590   &7.7 & 0.08085  & $25.9\%$ & $32.0\%$\\
			(5, 48)  &2 & 0.38484   &8.4 & 1.68508  & $23.7\%$ & $22.8\%$\\
			\bottomrule
		\end{tabular}
		\caption{\small  Comparison on Problem 4.}\label{tb:NM-INM-p4}
	\end{table}
}
{
	\begin{table}[h!]
		\centering
		\begin{tabular}{c|cccc|cc|ccc}\toprule
			\multirow{2}{*}{}
			& \multicolumn{4}{c|}{NM} & \multicolumn{2}{c|}{INM} & \multicolumn{3}{c}{Rate} \\
			$(m,n)$ & Iter& Time &Time-Int & R-Int & Iter & Time & RI & RT & RT1\\
			\hline
			(3,200)  &2.9 & 0.02955 & 0.02151  & $72.8\%$     &12 & 0.03493      & $24.7\%$ & $84.6\%$ & $23.0\%$ \\
			(3,401)  &2.9 & 0.35773 & 0.25802  & $72.1\%$     &12 & 0.44801      & $24.2\%$ & $79.8\%$ & $22.3\%$ \\
			(3,650)  &2.9 & 1.83363 & 1.38006  & $75.3\%$     &13 & 2.21081      & $22.5\%$ & $82.9\%$ & $20.5\%$ \\
			\hline
			(4, 40)  &2.9 & 0.00673 & 0.00474  & $70.4\%$     &9.1 & 0.00640      & $32.2\%$ & $105.1\%$ & $31.1\%$ \\
			(4, 71)  &2.8 & 0.09155 & 0.06100  & $66.6\%$     &9.6 & 0.14487      & $29.1\%$ & $63.2\%$ & $21.1\%$ \\
			(4,100)  &2.9 & 0.48084 & 0.31653  & $65.8\%$     &10 & 0.63916      & $28.3\%$ & $75.2\%$ & $25.7\%$ \\
			(4,130)  &2.9 & 1.53593 & 1.08875  & $70.9\%$     &11 & 1.90284      & $25.9\%$ & $80.7\%$ & $23.5\%$ \\
			\hline
			(5, 30)  &2.8 & 0.07459 & 0.04873  & $65.3\%$     &8.7 & 0.10647      & $32.3\%$ & $70.1\%$ & $24.3\%$ \\
			(5, 48)  &2.8 & 1.42363 & 0.92151  & $64.7\%$     &9.1 & 1.87335      & $30.3\%$ & $76.0\%$ & $26.8\%$ \\
			\bottomrule
		\end{tabular}
		\caption{\small  Comparison on Problem 5.}\label{tb:NM-INM-p5}
	\end{table}
}

It can be seen from Table 1 that the proposed NM has an advantage over the INM in \cite{Li-Guan-Xu-20} both in iteration and CPU time.
Paticularly, for Problem 1, 2 and 4, the coefficient tensor are diagonally dorminant,
and thus NM is easy to get an initial point, while for Problem 3 and 5,
although the number of iterations in NM was significantly less than in INM, the reduction in the required CPU time was not significant since it take much CPU time to find an initial point.

We then test the performance of Algorithm 3.4 (the Extended Newton Method denoted by 'ENM') by comparing with the Regularized Newton Method (denoted by 'RNM') proposed in \cite{Li-Guan-Xu-20}.
The parameter of ENM are set to $\epsilon=0.1, \epsilon^\prime=0.05,  \sigma=0.1$ and $\rho=0.5$.
And the parameters in RNM are the same as in \cite{Li-Guan-Xu-20}, i.e., $\sigma=0.1, \rho=0.8, \gamma=0.9$ and $\bar{t} = 0.01$.

For Problem 1, 2, 4 and 5, we generated $b\in R^{n}_{++}$ randomly and then let some but not all elements of $b$ be $0$ randomly, but we let $b_1 \neq 0$ in Problem 5.
Then it is easy to see that the strong M-tensor $\cal A$ and $b$ constructed in our tested problem satisfy the conditions in Assumption 3.1.
For Problem 5, we also use INM to find an initial point of ENM.
The parameter and starting point of INM are the same as before.
And the initial point of RNM is the same as INM.
The results are sumerized in Table \ref{tb:ENM-RNM-p1} to \ref{tb:ENM-RNM-p5}.

It can be seen from the data that the proposed ENM is not only effective for solving the M-Teq with $b\ge 0$, but also more efficient than RNM to a certain extent.
{
	\begin{table}[h!]
		\centering
		\begin{tabular}{c|cc|cc|cc}\toprule
			\multirow{2}{*}{}
			& \multicolumn{2}{c|}{ENM} & \multicolumn{2}{c|}{RNM} & \multicolumn{2}{c}{Rate} \\
			$(m,n)$ & Iter& Time & Iter & Time & RI & RT\\
			\hline
			(3,200)  &2.4 & 0.01286   &7.4 & 0.03194  & $32.3\%$ & $40.3\%$\\
			(3,350)  &2.3 & 0.05659   &7.2 & 0.18505  & $32.0\%$ & $30.6\%$\\
			(3,500)  &2.2 & 0.16599   &7.7 & 0.60120  & $28.1\%$ & $27.6\%$\\
			(3,650)  &2.2 & 0.43180   &7.8 & 1.38419  & $27.8\%$ & $31.2\%$\\
			\hline
			(4, 40)  &2.3 & 0.00467   &7.4 & 0.00732  & $31.1\%$ & $63.8\%$\\
			(4, 90)  &2.1 & 0.07785   &8 & 0.27286  & $26.4\%$ & $28.5\%$\\
			(4,130)  &2 & 0.38707   &8.1 & 1.20770  & $25.2\%$ & $32.1\%$\\
			\hline
			(5, 30)  &2.1 & 0.02927   &7.9 & 0.09237  & $26.2\%$ & $31.7\%$\\
			(5, 48)  &2 & 0.38698   &8.1 & 1.42453  & $24.6\%$ & $27.2\%$\\
			\bottomrule
		\end{tabular}
		\caption{\small  Comparison on Problem 1.}\label{tb:ENM-RNM-p1}
	\end{table}
}
{
	\begin{table}[h!]
		\centering
		\begin{tabular}{c|cc|cc|cc}\toprule
			\multirow{2}{*}{}
			& \multicolumn{2}{c|}{ENM} & \multicolumn{2}{c|}{RNM} & \multicolumn{2}{c}{Rate} \\
			$(m,n)$ & Iter& Time & Iter & Time & RI & RT\\
			\hline
			(3,200)  &3.5 & 0.01525   &9.6 & 0.04179  & $36.7\%$ & $36.5\%$\\
			(3,350)  &3.2 & 0.07255   &10 & 0.23803  & $32.4\%$ & $30.5\%$\\
			(3,500)  &3.2 & 0.22577   &10 & 0.75133  & $32.3\%$ & $30.0\%$\\
			(3,650)  &3.3 & 0.53518   &10 & 1.82259  & $32.6\%$ & $29.4\%$\\
			\hline
			(4, 40)  &3.4 & 0.00394   &9.6 & 0.00838  & $35.3\%$ & $47.0\%$\\
			(4, 90)  &3.3 & 0.10645   &10 & 0.40368  & $31.7\%$ & $26.4\%$\\
			(4,130)  &3.2 & 0.48901   &10 & 1.70984  & $30.9\%$ & $28.6\%$\\
			\hline
			(5, 30)  &3.3 & 0.03658   &10 & 0.12137  & $32.7\%$ & $30.1\%$\\
			(5, 48)  &3 & 0.52140   &11 & 2.20545  & $28.1\%$ & $23.6\%$\\
			\bottomrule
		\end{tabular}
		\caption{\small  Comparison on Problem 2.}\label{tb:ENM-RNM-p2}
	\end{table}
}
{
	\begin{table}[h!]
		\centering
		\begin{tabular}{c|cc|cc|cc}\toprule
			\multirow{2}{*}{}
			& \multicolumn{2}{c|}{ENM} & \multicolumn{2}{c|}{RNM} & \multicolumn{2}{c}{Rate} \\
			$(m,n)$ & Iter& Time & Iter & Time & RI & RT\\
			\hline
			(3,200)  &2.6 & 0.01099   &7.5 & 0.03303  & $34.5\%$ & $33.3\%$\\
			(3,350)  &2.2 & 0.05440   &7.6 & 0.19408  & $29.1\%$ & $28.0\%$\\
			(3,500)  &2.2 & 0.16818   &8 & 0.63763  & $27.4\%$ & $26.4\%$\\
			(3,650)  &2.1 & 0.40086   &8.1 & 1.64549  & $25.8\%$ & $24.4\%$\\
			\hline
			(4, 40)  &2.3 & 0.00477   &7.4 & 0.01114  & $30.6\%$ & $42.8\%$\\
			(4, 90)  &2 & 0.09399   &7.7 & 0.28420  & $26.6\%$ & $33.1\%$\\
			(4,130)  &2 & 0.35984   &8.1 & 1.39153  & $24.9\%$ & $25.9\%$\\
			\hline
			(5, 30)  &2 & 0.02607   &7.8 & 0.09134  & $26.0\%$ & $28.5\%$\\
			(5, 48)  &2 & 0.37042   &8.4 & 1.82102  & $24.1\%$ & $20.3\%$\\
			\bottomrule
		\end{tabular}
		\caption{\small  Comparison on Problem 4.}\label{tb:ENM-RNM-p4}
	\end{table}
}
{
	\begin{table}[h!]
		\centering
		\begin{tabular}{c|cccc|cc|ccc}\toprule
			\multirow{2}{*}{}
			& \multicolumn{4}{c|}{ENM} & \multicolumn{2}{c|}{RNM} & \multicolumn{3}{c}{Rate} \\
			$(m,n)$ & Iter& Time &Time-Int & R-Int & Iter & Time & RI & RT &RT1\\
			\hline
			(3,200)  &4.3 & 0.02343 & 0.01179  & $50.3\%$     &10 & 0.03949      & $42.8\%$ & $59.3\%$ & $29.5\%$ \\
			(3,350)  &4.2 & 0.14555 & 0.06743  & $46.3\%$     &10 & 0.24752      & $40.5\%$ & $58.8\%$ & $31.6\%$ \\
			(3,500)  &4.3 & 0.48284 & 0.21440  & $44.4\%$     &11 & 0.80104      & $39.7\%$ & $60.3\%$ & $33.5\%$ \\
			(3,650)  &4.1 & 1.10155 & 0.49533  & $45.0\%$     &11 & 1.87385      & $38.3\%$ & $58.8\%$ & $32.4\%$ \\
			\hline
			(4, 40)  &4.4 & 0.00624 & 0.00328  & $52.5\%$     &8.3 & 0.00636      & $52.5\%$ & $98.2\%$ & $46.6\%$ \\
			(4, 90)  &4.6 & 0.25343 & 0.10136  & $40.0\%$     &9.2 & 0.39225      & $49.5\%$ & $64.6\%$ & $38.8\%$ \\
			(4,130)  &4.4 & 1.13300 & 0.48959  & $43.2\%$     &9.5 & 1.64542      & $46.6\%$ & $68.9\%$ & $39.1\%$ \\
			\hline
			(5, 30)  &4.4 & 0.07540 & 0.03641  & $48.3\%$     &8 & 0.10592      & $54.4\%$ & $71.2\%$ & $36.8\%$ \\
			(5, 48)  &4.5 & 1.40943 & 0.60182  & $42.7\%$     &8.5 & 1.86548      & $53.2\%$ & $75.6\%$ & $43.3\%$ \\
			\bottomrule
		\end{tabular}
		\caption{\small  Comparison on Problem 5.}\label{tb:ENM-RNM-p5}
	\end{table}
}

\end{document}